\documentclass[12pt]{article}
\usepackage{natbib}
\usepackage[utf8]{inputenc}
\usepackage{amssymb,latexsym}
\usepackage{graphicx,amsmath,amsfonts,amsthm}
\usepackage{graphicx}
\newtheorem{lemma}{Lemma}
\newtheorem{proposition}{Proposition}

\newtheorem{theorem}{Theorem}
\newtheorem{example}{Example}
\newtheorem{definition}{Definition}
\newtheorem{corollary}{Corollary}
\newtheorem{criterion}{Criterion}
\newcommand{\cn}{{\mathcal N}}
\newcommand{\cd}{{\mathcal D}}
\newcommand{\cl}{{\mathcal L}}

\def\row#1#2{{#1}_1,\ldots ,{#1}_{#2}}

\begin{document}
\vspace*{12mm}
\renewcommand{\thefootnote}{*}
\begin{center}
{\LARGE{\bf A family of Condorcet domains that are single-peaked on a circle}}
\end{center}
\vspace{10mm}
\renewcommand{\thefootnote}{\arabic{footnote}}
\setcounter{footnote}{0}
\vspace{10mm}
\begin{center}
{\sc Arkadii Slinko}\vspace{3mm}\\
Department of Mathematics\vspace{1mm}\\
The University of Auckland\vspace{1mm}\\
Private Bag 92019\vspace{1mm}\\
Auckland 1142, New Zealand\vspace{1mm}\\
a.slinko@auckland.ac.nz 
\vspace{15mm}\\
\date{}
\vspace{12mm}
\end{center}
{\bf Abstract} {\em Fishburn's alternating scheme domains occupy a special place in the theory of Condorcet domains. Karpov (2023) generalised these domains and made an interesting observation proving that all of them are single-picked on a circle. However, an important point that all generalised Fishburn domains are maximal Condorcet domain remained unproved. We fill this gap and suggest a new combinatorial interpretation of generalised Fishburn's domains which provide a constructive proof of single-peakedness of these domains on a circle.  We show that classical single-peaked domains and single-dipped domains as well as  Fishburn's alternating scheme domains belong to this family of domains while single-crossing domains do not.}\par\vspace{12mm}
\noindent{\bf JEL Classification} D71\par\bigskip\medskip

\noindent{\bf Keywords:} Majority voting, transitivity, Condorcet domains, Fishburn's domains, single-crossing domain, single-peaked property, single-peaked property on a circle.  

\newpage 

\section{Preliminaries}

A Condorcet domain is a set of linear orders on a given set of alternatives such that, if all voters of a certain society are known to have preferences over those alternatives represented by linear orders from that set, the pairwise majority relation of this society is acyclic. 
Any background information on Condorcet domains can be found in survey papers \citep{mon:survey,slinko-puppe23}.

One of the best known Condorcet domains is the domain of single-peaked linear orders on a line spectrum \cite{black1958theory}. Recently \cite{peters2020preferences} generalised this domain to single-peaked domains on a circle. And, although so generalised domains are not necessarily Condorcet, as will be demonstrated in this paper, they have a role to play in the theory of Condorcet domains too. 

Intuitively, a domain is single-peaked on a circle if all the alternatives can be placed on a circle so that, for every order of the domain, we can `cut' the circle once so that the given order becomes single-peaked on the resulting line spectrum. The location of the cutting point may differ for different orders of the domain.\par\medskip

Let us briefly touch some of the basics of Condorcet domains. More information about them can be found in \cite{mon:survey,slinko-puppe23}.

By $\cl(A)$ we will denote all linear orders on the set of alternatives $A$ which will always be assumed to be finite. 
For a linear order $v\in \cl(A)$ and two alternatives $x,y\in A$ we write $x\succ_v y$ if $v$ ranks $x$ higher than $y$.  
The set of alternatives $A$ is often taken as $[n]=\{1,2,\ldots,n\}$.
 Up to an isomorphism, for $n=3$ there are only three maximal Condorcet domains:
\begin{equation*}
\label{eq:three_domains}
\cd_{1}=\{123, 312, 132, 321\},\ \cd_{2}=\{123,231, 132, 321\},\ \cd_{3}=\{123, 213, 231, 321\}. 
\end{equation*}

The domain $\cd_1$ on the left  contains all the linear orders on $[3]$
in which $2$ is never ranked first, the domain $\cd_2$ in the middle  
contains all the linear orders on $[3]$ in which $1$ is never ranked second, and domain $\cd_3$ on the right 
contains all the linear orders on $[3]$ in which $2$ is never ranked last.
Following \cite{mon:survey}, we denote these conditions as $2N_{\{1,2,3\}}1$, and
$1N_{\{1,2,3\}}2$, and $2N_{\{1,2,3\}}3$, respectively. 

\begin{definition}
\label{complete_set}
Any condition of type $xN_{\{a,b,c\}}i$ with $x\in \{a,b,c\}$ and $i\in \{1,2,3\}$ is called a {\em never condition} since it being applied to a domain $\cd$ requires  that in orders of the restriction $\cd|_{\{a,b,c\}}$ of $\cd$ to $\{a,b,c\}$ alternative $x$ never takes $i$th position. We say that a subset $\mathcal{N}$ of
\[
\{ xN_{\{a,b,c\}}i \mid \{a,b,c\}\subseteq A,\   x\in \{a,b,c\}\ \text{and $i\in \{1,2,3\}$} \}
\]
is a {\em complete set of never-conditions} if $\mathcal{N}$ contains exactly one never condition for every triple $a,b,c$ of elements of $A$. 
\end{definition}

The following criterion is a well-known characterisation of Condorcet domains that goes back to \cite{Sen1966}. See
also Theorem~1(d) in \cite{puppe2019condorcet} and references therein.

\begin{criterion}
\label{th:classic}
A domain of linear orders $\cd\subseteq \mathcal{L}(A)$ is a Condorcet domain
if and only if it satisfies a complete set of never conditions.
\end{criterion}

The following property of Condorcet domains was shown to be very important.

\begin{definition}[\cite{slinko2019}]
\label{def:copious}
A Condorcet domain $\cd$ is {\em copious} if for any triple of alternatives $a,b,c\in A$ the restriction $\cd|_{\{a,b,c\}}$ of this domain to this triple has four distinct orders, that is, $|\cd|_{\{a,b,c\}}|=4$.
\end{definition}

We note that, if a Condorcet domain is copious, then it satisfies a unique complete set of never conditions of the form~\eqref{complete_set}. Copiousness is often an important step in proving maximality of the domain. 

 \begin{proposition}
\label{cdcncopious}
Let $\cn$ be a complete set of never conditions and $\cd(\cn)$ is the set of all linear orders from $\cl(A)$ that satisfy $\cn$.  If $\cd(\cn)$ is copious, then $\cd(\cn)$ is a maximal Condorcet domain.
\end{proposition}

\begin{proof}
Suppose $\cd(\cn)$ is copious but not maximal. Then there exists a linear order $u$ such that $\cd'=\cd(\cn)\cup \{u\}$ is a Condorcet domain. Since $u\notin \cd(\cn)$ for a certain triple of alternatives $a,b,c$ the domain $\cd'|_{\{a,b,c\}}$ contains an order on $a,b,c$ which is not in $\cd|_{\{a,b,c\}}$. But then $\cd'|_{\{a,b,c\}}$ contains five orders on $a,b,c$ which is not possible as it would not be a Condorcet domain.
\end{proof}

Many Condorcet domains are defined relative to some sort of {\em societal axis}, also called {\em spectrum}. In politics it is often referred to as left-right spectrum of political opinions. 

\begin{definition}
A domain $\cd\subseteq \cl(A)$ is said to be (classical) single-peaked if there exists a societal axis (spectrum)
\[
a_1\triangleleft a_2\triangleleft \cdots\triangleleft a_n
\]
such that for every linear order $v\in \cl(A)$ and $a\in A$ the upper counter set $U(a,v)={\{b\in A \mid b\succ_v a\}}$ is convex relative to the spectrum. By $SP_n(\triangleleft)$ we will denote the domain of all single-peaked orders on $\triangleleft$.
\end{definition}

Up to an isomorphism $A$ is often taken as $[n]$ and the societal axis as $1<2<\cdots<n$.\par\medskip

Never conditions allow us to define useful classes of Condorcet domains as was pioneered by Peter \cite{PF:1997} who introduced the so-called alternating scheme of never conditions and constructed Condorcet domains of large order. \cite{karpov23} generalised his scheme as follows.

\begin{definition}[\cite{karpov23}]
Let $A=[n]$.
A complete set of never conditions is said to be a {\em generalised alternating scheme}, if for some subset $K\subseteq [2,\ldots,n-1]$ and
for all $1\le i<j<k\le n$ we have 
\begin{equation}
\label{eq:alt_scheme}
jN_{\{i,j,k\}}3,\ \text{if $j\in K$, and $ jN_{\{i,j,k\}}1$, if $j\notin K$}. 
\end{equation}
The domain which consists of all linear orders satisfying the generalised alternating scheme is called the {\em generalised Fishburn's domain} or {\em GF-domain}.
\end{definition}
The original Fishburn's alternating scheme has $K$ equal to the set of even numbers in ${[2,\ldots,n{-}1]}$. 
 The GF-domain constructed using a subset $K\subseteq [2,\ldots,n-1]$ will be denoted as $F_K$. Every GF-domain has  orders $12\ldots n$ and $n\ldots 21$ as they obviously satisfy conditions~\eqref{eq:alt_scheme}. Domains with this property are said to have {\em maximal width} \citep{Puppe2018}.

 \section{Combinatorial representation of GF-domains}
 
The idea of this  representation comes from an example in \cite{DKK:2010}.\par\medskip

A set of $n$ vertices on a circle, some white and some black, are numbered by integers $1,2,\ldots,n$. We will often identify the vertices with the numbers on them.

\begin{definition}
An arrangement of black and white vertices on a circle will be called a {\em necklace} and the vertices themselves will be called {\em beads}.
\end{definition}

\begin{definition}
\label{def:w-convex}
A set of beads $X\subseteq [n]$ is said to be {\em white convex} (or simply {\em $w$-convex}) if 
\begin{itemize}
\item $X$ is an arc of the circle;
\item $X$ does not consist of a single black bead;
\item There does not exist $i<j<k$ such that $i,k\in X$, $j\notin X$ and $j$ is white. 
\end{itemize}
\end{definition}

\begin{definition}
A {\em flag} of $w$-convex sets is a sequence $\row Xn$ of $w$-convex sets 
\begin{equation}
\label{eq:flag}
X_1\subset X_2\subset\cdots\subset X_n =[n],
\end{equation}
where $|X_k|=k$.
\end{definition}
Any flag \eqref{eq:flag} of $w$-convex sets defines a linear order $v=x_1x_2\ldots x_n$ on $[n]$, where $\{x_i\}=X_i\setminus X_{i-1}$ (for convenience we assume that $X_0=\emptyset$).

\begin{definition}
Given a necklace $S$ the domain $\cd(S)$ is the set of all linear orders corresponding to flags of $w$-convex sets.
\end{definition}

\begin{example}
Consider now the necklace $S$ presented on the picture on the left:
\begin{center}
\begin{minipage}{5cm}
\includegraphics[height=3.5cm]{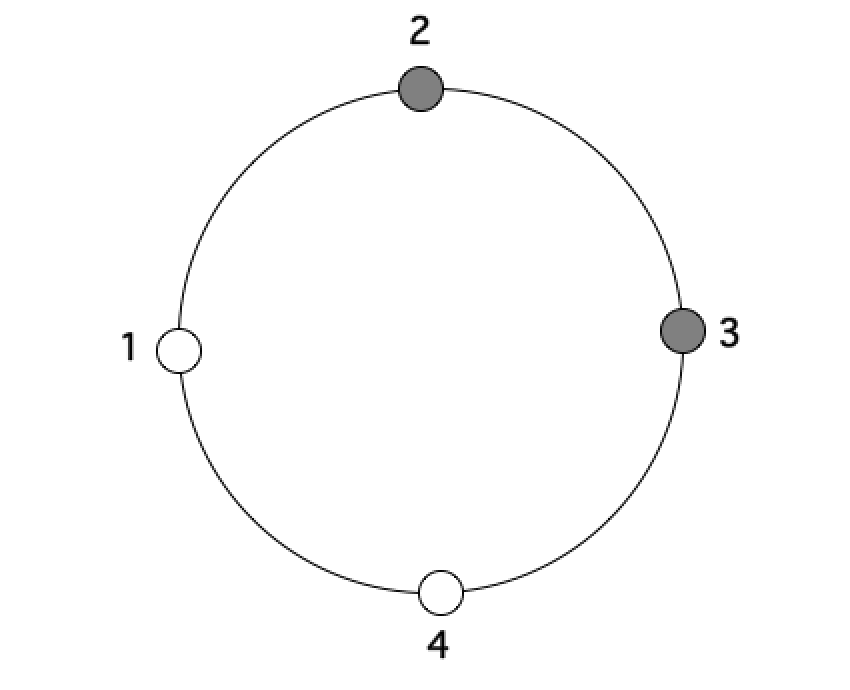}
\end{minipage}
\medskip
\begin{minipage}{5cm}
\[
\begin{array}{cccccccc}
1&1&1&1&4&4&4&4\\
2&2&4&4&1&1&3&3\\
3&4&2&3&2&3&1&2\\
4&3&3&2&3&2&2&1
\end{array}
\]
\end{minipage}
\end{center}
This is a single-dipped domain relative to the spectrum $1\triangleleft 2\triangleleft 3\triangleleft 4$ or $4\triangleleft 3\triangleleft 2\triangleleft 1$. 
\end{example}

\begin{example}
Consider now the necklace $S$ presented on the picture on the left:
\begin{center}
\begin{minipage}{5cm}
\includegraphics[height=3.5cm]{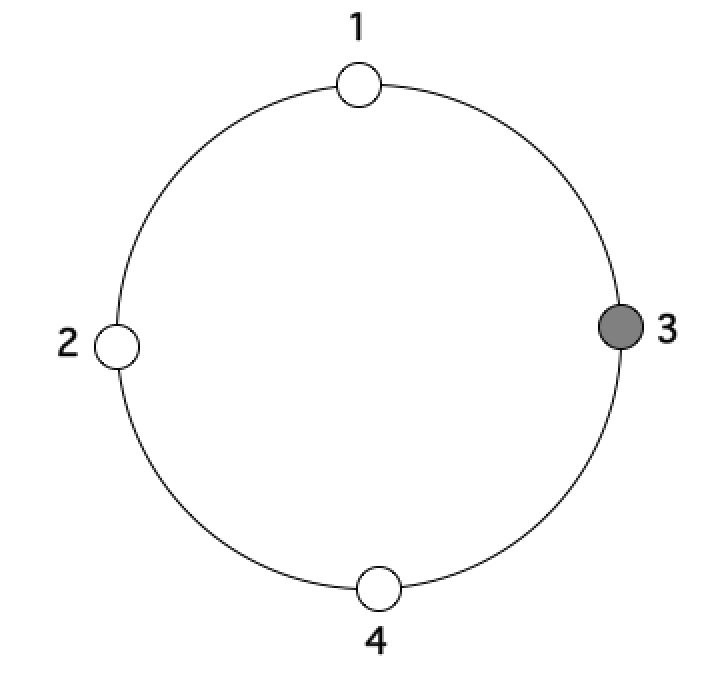}
\end{minipage}
\medskip
\begin{minipage}{5cm}
\[
\begin{array}{ccccccccc}
1&1&2&2&2&2&4&4&4\\
2&2&1&1&4&4&2&2&3\\
3&4&3&4&1&3&1&3&2\\
4&3&4&3&3&1&3&1&1
\end{array}
\]
\end{minipage}
\end{center}
Then domain $\cd(S)$ is given by the array on the right.
This is the Fishburn domain relative to the spectrum $1\triangleleft 2\triangleleft 3\triangleleft 4$. 
\end{example}
 
Our example shows that the construction is promising and generating maximal  GF-domains. Let us generalise these examples and offer a new combinatorial representation of GF-domains from which we will deduce their maximality.   \par\medskip

Let $K\subseteq [2,\ldots,n-1]$ and ${L=[2,\ldots,n-1]\setminus K}$ be two complementary subsets of $[2,\ldots,n-1]$. Let $k_1<\ldots<k_s$ and $\ell_1<\ldots<\ell_t$ be ordered elements of $K$ and $L$, respectively, where $s+t=n-2$. Consider the following spectrum on the circle 
 \begin{equation}
 \label{eq:spectrum}
 1\triangleleft k_1\triangleleft\ldots\triangleleft k_s\triangleleft n\triangleleft \ell_t\triangleleft \ldots\triangleleft \ell_1\triangleleft 1.
 \end{equation}
 Mark beads $1,\row ks,n$ white and $\row{\ell}t$ black to obtain a necklace $S_K$.
 
 \begin{example}
  For $n=3$ we have two options: one with $K_1=\emptyset$ and another with $K_2=\{2\}$. Respectively we have two necklaces $S_{K_1}$ and $S_{K_2}$:
 \begin{center}
\includegraphics[height=5cm]{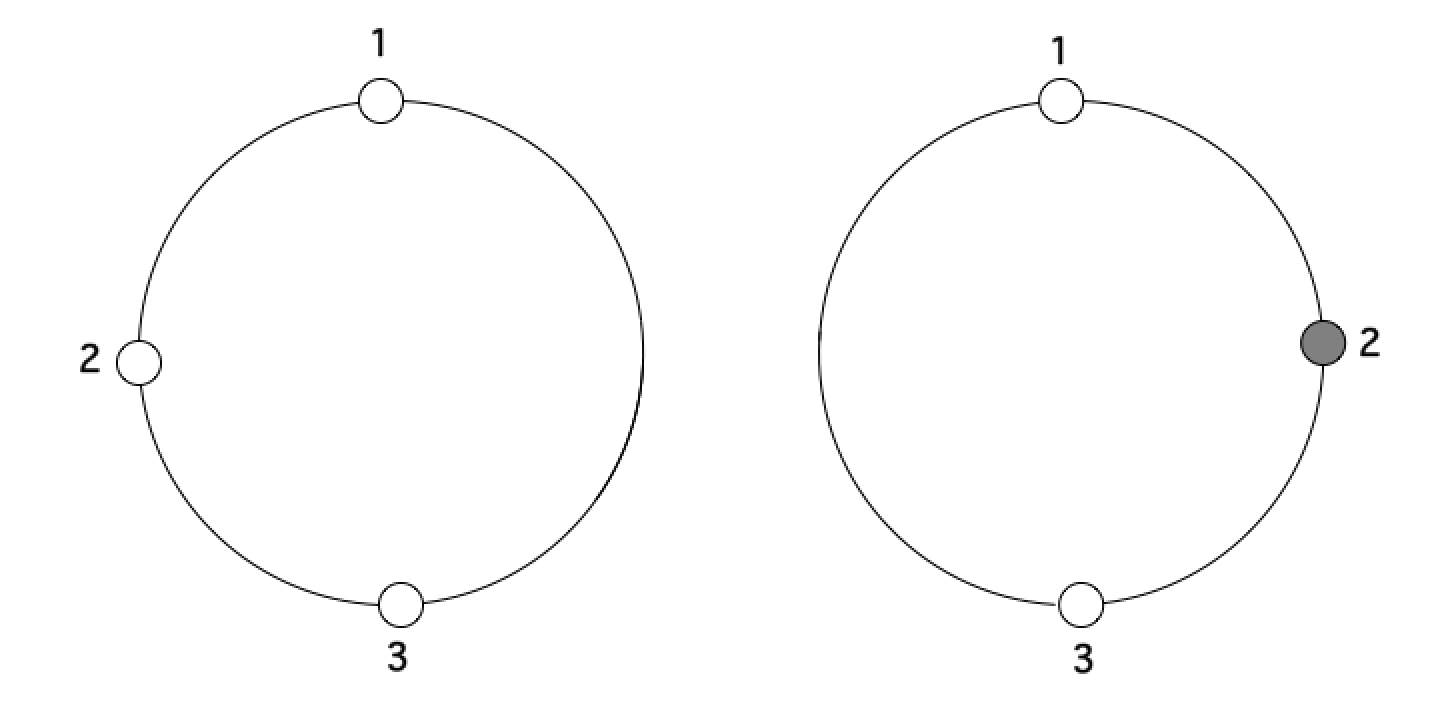}
\end{center}
Then $\cd(S_{K_1})=\{123, 213, 231, 321\}$ and $\cd(S_{K_2})=\{123, 132, 312, 321\}$ which are $F_{K_1}$ and $F_{K_2}$, respectively. These are single-peaked and single-dipped triples.
 \end{example}
 
  \begin{proposition}
 If $K=[2,\ldots,n-1] $, then $\cd(S_K)=F_K$ is the classical single-peaked domain.
 \end{proposition}
 
 \begin{proof}
 Since $L=\emptyset$ in $S_K$ there are no black beads and beads 1 and $n$ are neighbours on the circle. The only $w$-convex set containing both of them is the longer arc $Z$ with endbeads 1 and $n$, that is, $\{1,2,\ldots,n\}$.  Any arc $X\subseteq Z$ is $w$-convex.  So $w$-convex sets coincide with the upper contour sets of the classical single-peaked domain with the spectrum  $1\triangleleft 2\triangleleft\ldots\triangleleft n-1\triangleleft n$.
 \end{proof}
 
\begin{lemma}
\label{dom_S_K}
 In domain $\cd(S_K)$:
 \begin{itemize}
 \item[(i)] A black bead $a$ satisfies the never-top condition in any triple $\{a,b,c\}$ such that $b<a<c$, that is $aN_{\{a,b,c\}}1$. 
 \item[(ii)] A white bead $a$ satisfies the never-bottom condition in any triple $\{a,b,c\}$ such that $b<a<c$, that is $aN_{\{a,b,c\}}3$. 
 \end{itemize}
 \end{lemma}
 
 \begin{proof}
(i) Suppose $a\in L\subseteq [2,\ldots,n]$ is black and there is a linear order $v$ in $\cd(S_K)$ whose restriction to subset $\{a,b,c\}$ is $abc$ with $b<a<c$. Let $X$ be the first $w$-convex set from the flag corresponding to $v$ that contains $a$. Then $a=\ell_i$ and $X=\{\ell_{i},\ldots, \ell_1, 1, k_1,\ldots, k_j\}$  does not contain $b$ and $c$ (here it is possible that $j=0$). We note that $b\notin L$ as $b<\ell_i$ and not in $X$, thus $b=k_s$ with $s>j$. But then $1<k_s<\ell_i$ and $X$ is not $w$-convex. 

(ii) Suppose $a\in K\subseteq [2,\ldots,n]$ white and a certain linear order $v$ in $\cd(S_K)$ has restriction $bca$ to subset $\{a,b,c\}$ is  with $b<a<c$. Let $X$ be the largest $w$-convex set in the  flag corresponding to $v$ that still does not contain $a$. Then it contains $b$ and $c$ which contradicts to $w$-convexity of $X$.
 \end{proof}
 
  \begin{theorem}
  \label{SK=FK}
 $\cd(S_K)= F_K$.
 \end{theorem} 
 
 \begin{proof}
 By Lemma~\ref{dom_S_K} we know that  $\cd(S_K)\subseteq  F_K$. Let us prove the converse. 
 Let $v=a_1\ldots a_n\in F_K$. We define the $k$-th ideal of $v$ as $\text{Id}_k(v)=\{a_1,\ldots, a_k\}$. It is enough to show that for any $k\in[n]$ the set $\text{Id}_k(v)$ is $w$-convex. Suppose not. Then there exist $a,b,c\in [n]$ with $a,c\in \text{Id}_k(v)$ and $b\notin \text{Id}_k(v)$ satisfying $a<b<c$ and $b\in K$ is white. Then the restriction of $v$ onto $\{a,b,c\}$ is $acb$ or $cab$ in violation of $bN_{\{a,b,c\}}3$. 
 This contradiction proves the theorem.
 \end{proof}

\begin{lemma}
\label{copiuos}
For any $K\subseteq [2,\ldots, n-1]$ the domain $F_K$ is copious.
\end{lemma}
 
 \begin{proof}
 We will use Theorem~\ref{SK=FK} and consider $\cd(S_K)$ instead.
 Let $a,b,c\in [n]$ with $a<b<c$. 
 We need to consider several cases. 
 \begin{enumerate}
 \item $a,b,c$ are all white. Then $a=k_p$, $b=k_s$, $c=k_r$ with $p<s<r$. The following sets are $w$-convex:
 \[
 \{k_p\},  \{k_s\},  \{k_r\}, \{k_p,\ldots,k_s\}, \{k_s,\ldots, k_r\}, \{k_p,\ldots,k_s,\ldots,k_r\}.
 \] 
 Thus, $abc, cba, bac, bca$ all belong to the restriction of $S_K$ onto $\{a,b,c\}$.
\item $a,b,c$ are all black. Suppose $a=\ell_p$, $b=\ell_q$, $c=\ell_r$ with $p<q<r$. Let $K'=K\cup\{1\}\cup \{n\}$.
Note that every arc containing $K$ is $w$-convex. 
Then we have the following $w$-convex sets: 
$
K'\cup \{\ell_t,\ldots,\ell_r\}\subset K'\cup \{\ell_t,\ldots,\ell_q\}\subset K'\cup \{\ell_t,\ldots,\ell_p\}, 
$
hence the restriction of $S_K$ onto $\{a,b,c\}$ contains $cba$. Another sequence of $w$-convex sets  
$
K'\cup \{\ell_1,\ldots,\ell_p\}\subset K'\cup \{\ell_t,\ldots,\ell_q\}\subset K'\cup \{\ell_t,\ldots,\ell_r\}
$
gives us $abc$. Now, let us consider $K''=K'\cup \{\ell_1,\ldots,\ell_{p-1}\}\cup \{\ell_t,\ldots,\ell_{r-1}\}$. Then the sequence of 
$
K''\cup \{\ell_r\}\subset K''\cup \{\ell_r, \ell_p\}\subset [n]
$
gives us $cab$ and the sequence
$
K''\cup \{\ell_p\}\subset K''\cup \{\ell_r, \ell_p\}\subset [n]
$
gives us $cba$. Hence we have four suborders in $S_K|_{\{a,b,c\}}$.

 \item  $a$ is white; $b,c$ are black. Then obviously, $abc$ and $acb$ belong to the restriction of $S_K$ onto $\{a,b,c\}$. But also in  the restriction of $S_K$ onto $\{n,a,b,c\}$ we have $ncba$ and $ncab$, hence $cba$ and $cab$ belong to $S_K|_{\{a,b,c\}}$, so this restriction has four suborders.  
 
 \item  $b$ is white; $a,c$ are black. Then $bac$ and $bca$ are in $S_K|_{\{a,b,c\}}$ as well as $ncba$ and $1abc$  (or $1cba$ and $nabc$ belong to the restrictions of $S_K$ onto $\{n, a,b,c\}$ and $\{1, a,b,c\}$, respectively. Hence $cba$ and $abc$ are in $S_K|_{\{a,b,c\}}$ and this restriction has four suborders as well.  
 
 \item $a$ is black; $b,c$ are white. Then $bca$, $cba$, $bac$ and $1abc$ belong to respective restrictions, so four suborders. 
 
 \item If $a$ and $b$ are black and $c$ is white. Then $cba$ and $cab$ belong to $S_K|_{\{a,b,c\}}$ together with $1bca$ and $1bac$.  These are all possible cases.\qedhere
 \end{enumerate}
 \end{proof}

Combining Proposition~\ref{cdcncopious} with Lemma~\ref{copiuos} we get

\begin{theorem}
\label{FKmax}
For any $K\subseteq [2,\ldots,n-1]$ the domain $F_K$ is a maximal Condorcet domain.
\end{theorem}


 
The universal domain $\cl (A)$ has many representations. One of the most useful ones is by the {\em permutohedron} of order $n$ \citep{mon:survey}, 
whose vertices  are labeled by the permutations of $[n]$ from the symmetric group $S_n$. Two vertices are connected by an edge if their permutations differ in only two neighbouring places. 
Domains can be considered as a subgraphs of the permutohedron.

\begin{definition}
\label{def:sem_con}
A domain $\cd$ of maximal width is called {\em semi-connected} if the two completely reversed orders
$e$ and $\bar{e}$ from $\cd$ can be connected by a shortest path (geodesic path) in the permutohedron whose all vertices belong to $\cd$. 
%
It is {\em directly connected}, if any two orders of a domain are connected by a shortest path in the permutohedron that stays within the domain.
\end{definition}

Maximality of GF-domains has a number of profound consequences.

\begin{theorem} 
Every GF-domain $\cd$ is a directly connected domain of maximal width.
\end{theorem}

\begin{proof}
We have already noticed that GF-domains have maximal width containing $12\ldots n$ and $n\ldots 21$. By their definition, they are also the so-called peak-pit domains which means that they satisfy a complete set of never-top and never-bottom conditions. By Theorem~2  of \cite{DKK:2012} maximality of $\cd$ implies that this is a tiling domain and, in particular, it is semi-connected.  It has been observed in \cite{Puppe16} (Proposition A.1) that maximal semi-connected domains are   directly connected.
\end{proof}

Let us now give a formal definition of a domain single-peaked on a circle.

\begin{definition}[\cite{peters2020preferences}]
A linear order $v\in \cl(A)$ is said to be {\em single-peaked on a circle}, if alternatives from $A$ can be placed on a circle 
\[
a_1\triangleleft a_2\triangleleft \cdots\triangleleft a_n\triangleleft a_1
\]
in anticlockwise order so that for every alternative $a\in A$ the upper counter set $U(a,v)=\{b\in A \mid b\succ_v a\}$ is a contiguous arc of the circle. 

A domain $\cd\subseteq \cl(A)$  is said to be {\em single-peaked on a circle} if there exists an arrangement of alternatives on that circle such that each order of $\cd$ is single-peaked on a circle relative to their common arrangement of alternatives. 
\end{definition}
Our Theorem~\ref{SK=FK} as a corollary provides a constructive proof of the following theorem.
 
  \begin{corollary}[\cite{karpov23}]
 \label{thm:SPOC}
 Every GF-domain $F_K$ is single-peaked on a circle.
 \end{corollary}
 
 \begin{proof}
By Theorem~\ref{SK=FK}  $F_K$ is isomorphic to $\cd(S_K)$. The statement now follows from the fact that any upper contour set of this domain is a contiguous arc of the necklace.
\end{proof}

Original Karpov's proof was based on the characterisation of single-peaked on a circle domains by means of forbidden configurations given in \cite{peters2020preferences}.\par\medskip

The question may be asked: Are all peak-pit maximal Condorcet domains of maximal width are single-peaked on a circle? The answer is negative. 

\begin{theorem}
For any $n\ge 4$ single-crossing maximal Condorcet domain are not single-peaked on a circle.
\end{theorem}

\begin{proof}
\cite{slinko2021characterization} characterised all single-crossing maximal Condorcet domains in terms of the relay structure. In this structure linear orders are arranged in a sequence $v_1,v_2,v_3,\ldots$, so that moving from left to right 1 initially moves from top to bottom being swapped with $2, 3, 4, \ldots$. The sequence of the first three swapping pairs is:
$
(1,2),\ (1,3),\ (1,4)\ldots. 
$
The three linear orders $v_1,v_2,v_3$ have $\text{Id}_2(v_1)=\{1,2\}$,  $\text{Id}_2(v_2)=\{1,3\}$ and $\text{Id}_2(v_3)=\{1,4\}$. But It is impossible to have such three arcs on a circle. 
\end{proof}

In this proof effectively we spotted in any single-crossing maximal Condorcet domain one of the forbidden configurations described in~\cite{peters2020preferences}.

\section{Conclusion and future work}

Now we know that the single-peaked domain, Fishburn's domain and single-dipped domain are members of the same family with single-peaked and single-dipped being the two extremes.
It would be interesting to investigate how the size of domain $F_K$ depends on $K$. It is well-known that when $K=[2,\ldots,n-1]$ or $K=\emptyset$ (the case of classical single-peaked and single-dipped domains) we have $|F_K|=2^{n-1}$ and when $K$ is the set of even numbers in $[2,\ldots,n-1]$ (the case of classical Fishburn's domain)  \cite{GR:2008} gave the exact formula for the cardinality of $F_K$:
\begin{equation*}
\label{card_alt_s}
|F_K|= (n+3)2^{n-3} - 
\begin{cases}
(n-\frac{3}{2}){n-2\choose \frac{n}{2}-1} & \text{for even $n$;}\\
(\frac{n-1}{2}){n-1\choose \frac{n-1}{2}} & \text{for odd $n$.}
\end{cases}
\end{equation*} 
It is reasonable to conjecture that these are the most extreme cases and the cardinality of $|F_K|$ for various $K$ must be somewhere in between. However, since Fishburn's domain is not the largest peak-pit domain of maximal width for at least $n\ge 34$ \citep{karpov2023constructing},  we do not even know if Fishburn's domains are the largest among all GF-domains. 
 
 \section{Acknowledgements}
 
 The author thanks A. Karpov for bringing the paper \cite{DKK:2010} to his attention.

 \bibliographystyle{plainnat}
\bibliography{cps-survey}
\end{document}